\RequirePackage[l2tabu, orthodox]{nag}

\documentclass[12pt]{amsart}
\usepackage{xypic}

\usepackage[usenames,dvipsnames]{xcolor}
\definecolor{amethyst}{rgb}{0.6, 0.4, 0.8}

\usepackage[pagebackref=true, colorlinks, allcolors=amethyst]{hyperref}

\renewcommand*{\backref}[1]{}
\renewcommand*{\backrefalt}[4]{%
  \ifcase #1 %
    \relax
  \or
    $\uparrow$#2.%
  \else
    $\uparrow$#2.%
  \fi%
}

\usepackage{mathtools}

\usepackage{fullpage,url,amssymb,enumitem,colonequals,algorithm2e}

\usepackage{mathrsfs} 

\usepackage{tabularx} 
\usepackage{booktabs} 

\usepackage{breakurl} 

\usepackage[OT2,T1]{fontenc}
\DeclareSymbolFont{cyrletters}{OT2}{wncyr}{m}{n}
\DeclareMathSymbol{\Sha}{\mathalpha}{cyrletters}{"58}
\usepackage{tikz-cd,tikz}

\newcommand{\defi}[1]{\textsf{#1}} 

\newcommand{\F}{\mathbb{F}}

\newcommand{\PP}{\mathbb{P}}
\newcommand{\Q}{\mathbb{Q}}

\newcommand{\Z}{\mathbb{Z}}

\newcommand{\Mbar}{{\overline{M}}}

\newcommand{\calO}{\mathcal{O}}

\DeclareMathOperator{\Ann}{Ann}

\DeclareMathOperator{\Hom}{Hom}

\DeclareMathOperator{\rank}{rank}

\DeclareMathOperator{\Sym}{Sym}
\DeclareMathOperator{\tr}{tr}

\newcommand{\tors}{{\operatorname{tors}}}

\newcommand{\tensor}{\otimes} 

\newtheorem{theorem}{Theorem}[section]
\newtheorem{lemma}[theorem]{Lemma}
\newtheorem{corollary}[theorem]{Corollary}
\newtheorem{proposition}[theorem]{Proposition}

\theoremstyle{definition}
\newtheorem{definition}[theorem]{Definition}

\newtheorem{example}[theorem]{Example}

\newtheorem{myalgorithm}[theorem]{Algorithm}

\theoremstyle{remark}
\newtheorem{remark}[theorem]{Remark}

\usepackage{microtype}

\renewcommand{\angle}[1]{\hspace{-2pt}\left\langle #1 \right\rangle}
\newcommand{\AJ}{\mathrm{AJ}}
\newcommand{\GC}{\mathrm{GLC}}
\newcommand{\CC}{\mathrm{CC}}

\title{A geometric linear Chabauty comparison theorem}
\author{Sachi Hashimoto and Pim Spelier}
\thanks{SH was supported by National Science Foundation grant DGE-1840990.}
\begin{document}
\begin{abstract}
The Chabauty--Coleman method is a $p$-adic method for finding all rational points on curves of genus $g$ whose Jacobians have Mordell--Weil rank $r < g$. Recently, Edixhoven and Lido developed a geometric quadratic Chabauty method that was adapted by Spelier to cover the case of geometric linear Chabauty. We compare the geometric linear Chabauty method and the Chabauty--Coleman method and show that geometric linear Chabauty can outperform Chabauty--Coleman in certain cases. However, as Chabauty--Coleman remains more practical for general computations, we discuss how to strengthen Chabauty--Coleman to make it theoretically equivalent to geometric linear Chabauty. We apply these methods to genus $2$ and genus $3$ curves.
\end{abstract}

\maketitle

\section{Introduction}

Let $C_{\Q}/\Q$ be a smooth, proper, geometrically integral curve of genus $g \geq 2$. Faltings's theorem \cite{Faltings} states that the set of rational points $C_{\Q}(\Q)$ is finite. However, it does not provide an explicit method for computing this finite set. Let $J_{\Q}/\Q$ be the Jacobian of $C_{\Q}$, with Mordell--Weil rank $r$. Fix a prime $p>2$ of good reduction for $C_{\Q}$. The Chabauty--Coleman method is an explicit $p$-adic method for computing the set of rational points on $C_{\Q}$ when $r<g$. Letting $C$ be a model of $C_\Q$ over $\Z_{(p)}$, the method computes a finite set of $p$-adic points $C(\Z_p)_\CC$ containing the rational points $C(\Z_{(p)}) = C_{\Q}(\Q)$.

In recent years, the Chabauty--Coleman method has been extended to lift the restriction $r<g$; Balakrishnan, Besser, M\"uller, Dogra, Tuitman, and Vonk \cite{QCIntegral,QCI,QCII,QCCartan} developed the quadratic Chabauty method. Edixhoven and Lido proposed a parallel geometric quadratic Chabauty method \cite{EdixhovenLido} that uses algebro-geometric methods and works in torsors over the Jacobian instead of a certain Selmer variety.

Spelier \cite{PimThesis} adapted the geometric method in Edixhoven--Lido to the linear case of Chabauty--Coleman. They outlined a theory of geometric linear Chabauty that parallels the Chabauty--Coleman method. This geometric method works in the Jacobian itself instead of its image under the logarithm in $\Q_p^g$.

This idea of working in the Jacobian itself is not new. Previously, Flynn \cite{flynn97} also leveraged the Jacobian group law to perform Chabauty-type calculations in a similar way to our geometric method. Flynn's method relied on explicit equations and explicit group laws for $J_{\Q}$ in high-dimensional projective space; the method was used to compute the rational points in several new cases of genus $2$ and Mordell--Weil rank $1$ curves. However, some of the ideas in \cite{flynn97} do not generalize that easily to higher genus and higher rank examples. Flynn uses specific equations for the embedding of genus $2$ curves in $\PP^{15}$ and theorems on the number of zeros of $p$-adic univariate polynomials that both do not extend easily to generic higher genus Jacobians and higher Mordell--Weil ranks.

While geometric linear Chabauty sacrifices the explicit nature of Flynn's method, it can nevertheless be applied to curves of any genus.
The geometric linear Chabauty method computes a finite set of $p$-adic points $C(\Z_p)_\GC$ containing the set of rational points $C_\Q(\Q)$. The method can be performed modulo $p^n$ for any precision $n \in \Z_{>0}$, although it does not always result in an upper bound on the number of rational points. Done modulo $p$, the computations are simply linear algebra.

In this paper, we survey both the geometric linear Chabauty and Chabauty--Coleman methods, and we provide many examples of the new geometric linear Chabauty method of Spelier. Our main result is a comparison theorem between the two methods. In Theorem \ref{mainthmcomp}, we show that the geometric linear Chabauty method outperforms the Chabauty--Coleman method in certain cases. We have the inclusions: \[ C_\Q(\Q) \subseteq C(\Z_p)_\GC \subseteq C(\Z_p)_\CC.\] Furthermore we give an explicit characterization of any excess points the Chabauty--Coleman method finds, i.e. of the set $C(\Z_p)_\CC \setminus C(\Z_p)_\GC$.

However, because the geometric linear Chabauty method can be prohibitively difficult to implement, in Algorithm \ref{alg:improvedCC} we instead provide an upgrade for the Chabauty--Coleman method that makes it equivalent to geometric linear Chabauty. Finally, this paper makes a practical improvement to the geometric linear Chabauty method, replacing complicated Jacobian arithmetic over $\Z/p^2\Z$ with very low precision Coleman integration on the curve and arithmetic in $\F_p^g$.

We start by defining our notational conventions in Section \ref{set-up}.
In Section \ref{GeoChab} we introduce the geometric linear Chabauty method of Spelier.  Section \ref{S:ChabCol} reviews the Chabauty--Coleman method. We showcase the explicit linear algebra method for finding rational points on $C_{\Q}$ in Section \ref{S:ExplicitModP}. The main theorem and discussion on comparison is found in Section \ref{comparison}.

\section{Acknowledgments}
We are very grateful to Jennifer Balakrishnan for helpful comments during the preparation of this paper. We are also thankful to Bas Edixhoven for his generous advice and Steffen M\"uller for assistance with Magma computations. We thank the anonymous referee for many helpful suggestions for improving the paper.
\section{Background}
\label{background}  
\subsection{Set-up} 
\label{set-up}

Let $C_{\Q}/\Q$ be a smooth, proper, geometrically integral curve of genus $g \geq 2$.  Fix $p>2$ a prime of good reduction for $C_{\Q}$ and let $C /{\Z_{(p)}}$ be a smooth model for the curve over the local ring. Then
\begin{equation}C(\Z_{(p)}) = C_{\Q}(\Q) \end{equation}
so the problem of determining $\Q$-points on $C_\Q$ can be replaced by the problem of determining $\Z_{(p)}$-points on $C$.

Let $J/\Z_{(p)}$ be the Jacobian of $C$ and suppose that the Mordell--Weil rank $r$ of $J_\Q(\Q)$ or, equivalently, $J(\Z_{(p)})$ is less than $g$. We use $M$ to denote the $p$-adic closure of the Mordell--Weil group $\overline{J(\Z_{(p)})}$ in $J(\Z_p)$. Denote the torsion subgroup of $M$ by $M^\tors$. Let $r' \leq r$ be the rank of $M/M^\tors$ as a $\Z_p$-module; we assume we have computed $r'$ elements of $J(\Z_{(p)})$ that topologically generate $M$. We also assume $C(\Z_{(p)})$ is non-empty and fix forever a basepoint $b \in C(\Z_{(p)})$. Let $\overline{b} \in C(\F_p)$ denote the reduction of $b$ modulo $p$.

\begin{remark}
In the case that $r$ is at most $g$, one usually has $r' = r$. If $r' < r$, there is generally a geometric reason for this, for example the Jacobian splitting as a product of smaller abelian varieties up to isogeny, in which case $r'$ itself and $r'$ topological generators can often be computed if the Mordell--Weil group is known.
\end{remark}

For $X$ a scheme, $R$ a local ring with residue field $\F_p$, and $Q \in X(\F_p)$, let $X(R)_Q$ denote the residue disk $\{x \in X(R) : \bar{x} = Q\}$ over $Q$; we use the same notation for the residue disks of $M$.

We will need a description of $X(R)_Q$ the residue disk of a smooth scheme over $\Z_p$. For this, we use the following lemma from \cite{PimThesis} that can be applied to an affine chart of $X$ containing $Q$.
\begin{lemma}[{\cite[Lemma~2.2]{PimThesis}}]
\label{lem:smooth}
Let $X$ be a smooth affine scheme over $\Z_p$ of relative dimension $d$, let $Q$ be an $\F_p$-point of $X$, and let $t_1,\dots,t_d$ be parameters of $X$ at $Q$, i.e. elements of the local ring $\calO_{X,Q}$ such that the maximal ideal is given by $(p,t_1,\dots,t_d)$. Define $\widetilde{t}_i \colonequals t_i/p$.
Then evaluation of $\widetilde{t}$, the vector $(\widetilde{t}_1,\dots,\widetilde{t}_d)$, gives a bijection $\widetilde{t}\colon X(\Z_p)_Q \to (\Z_p)^d$.
\end{lemma}
In fact, this is shown in a geometric fashion by giving a bijection between $X(\Z_p)_Q$ and $\widetilde{X}_Q^p(\Z_p)$, an open affine subscheme of the blowup of $X$ at $Q$. Then the coordinate ring of $\widetilde{X}_Q^p(\Z_p)$ has $p$-adic completion equal to the ring of convergent power series 
\[\Z_p\angle{\widetilde{t}_1,\dots,\widetilde{t}_d} = \{f \in \Z_p[[\widetilde{t}_1,\dots,\widetilde{t}_d]] : \text{ for all } n \geq 0, f \in \Z_p[\widetilde{t}_1,\dots,\widetilde{t}_d] + (p^n) \}.\]
 Evaluating the $\widetilde{t_i}$ yields a bijection $\widetilde{X}_Q^p(\Z_p) \to \Z_p^d$ by the formula \[\widetilde{X}_Q^p(\Z_p) =  \Hom( \calO_{\widetilde{X}_Q^p}, \Z_p) = \Hom(\Z_p\angle{\widetilde{t}_1,\dots,\widetilde{t}_d} , \Z_p) = \mathbb{A}_{\Z_p}^d.\] 

\begin{remark}
Lemma \ref{lem:smooth} works equally well modulo $p^n$, giving a bijection \[X(\Z/p^n\Z)_Q \simeq (\Z/p^{n-1}\Z)^d.\]
\end{remark}

\subsection{The geometric linear Chabauty method}
\label{GeoChab}

We recall an idea of Chabauty proving the finiteness of rational points on certain curves of genus $g\geq 2$.
\begin{theorem}[\cite{ChabautyBetter}]
\label{Chabauty}
Let $\AJ_b\colon C(\Z_p) \to J(\Z_p)$ denote the Abel--Jacobi map induced by the basepoint $b$. Then $\AJ_b (C(\Z_p)) \cap M$ is finite and therefore $C(\Z_{(p)})$ is.
\end{theorem}
The geometric linear Chabauty method makes Theorem~\ref{Chabauty} explicit by computing the set $\AJ_b (C(\Z_p)) \cap M$ exactly. To start, we break up the set into a union of residue disks. Fix $Q \in C(\F_p)$ and consider the set $\AJ_{b}( C(\Z_p)_Q) \cap M_{Q-\overline{b}}$ in $J(\Z_p)$. We study the closure of the Mordell--Weil group and the image of the curve under Abel--Jacobi separately.

To describe $M_{Q-\overline{b}}$, we simply need to know whether it is empty; if not, fix a choice of $T \in J(\Z_{(p)})_{Q-\overline{b}}$; if it is, then $\AJ_{b} (C(\Z_p)_Q) \cap M_{Q-\overline{b}}$ is empty, so we can sieve out this residue disk. Indeed, this is a reformulation of the Mordell--Weil sieve at the prime $p$, as discussed in Section \ref{subs:sieve}.

Now we identify $J(\Z_p)_{Q-\overline{b}}$ with $\Z_p^g$ by Lemma \ref{lem:smooth}. Note that this identification does not preserve the additive structure. Then $\AJ_b (C(\Z_p)_Q)$ is cut out by convergent power series $f_1, \dots, f_{g-1}$ in the ring of convergent $p$-adic power series $\Z_p \langle z_1, \dots, z_g \rangle$ \cite[Remark~2.6]{PimThesis} (for the right choice of parameters, we may assume the $f_i$ are linear), and the inclusion $M_{Q-\overline{b}} \to J(\Z_p)_{Q-\overline{b}}$, identifying the former with $\Z_p^{r'}$, is given by $g$ power series $\kappa_1,\dots,\kappa_g \in \Z_p \langle x_1, \dots, x_{r'} \rangle$ \cite[Theorem~3.1]{PimThesis}. All in all, we get the diagram

\begin{equation}\label{diag:GC}\begin{tikzcd}
  && {0} \\
  && {\Z_p^{g-1}} \\
  {0} & {M_{Q-\overline{b}}} & {J(\Z_p)_{Q-\overline{b}}} \\
  && {C(\Z_p)_Q} \\
  && {0}
  \arrow["{\kappa}", from=3-2, to=3-3]
  \arrow["{f}", from=3-3, to=2-3]
  \arrow["{\AJ_b}", from=4-3, to=3-3]
  \arrow[from=3-1, to=3-2]
  \arrow[from=2-3, to=1-3]
  \arrow[from=5-3, to=4-3]
  \arrow["{\lambda_{\GC}^Q}", from=3-2, to=2-3]
\end{tikzcd}.
\end{equation}

The coordinates $\lambda_i$ for $i=1,\dots,g-1$ of $\lambda_{\GC}^Q$ consist of the pullbacks of $f_1,\dots,f_{g-1}$ along $\kappa$. They are given by composing convergent power series that are affine linear mod $p$ and thus themselves are given by convergent power series that are affine linear mod $p$. In this diagram, the vertical sequence is exact in the sense that $f^{-1}(0) = \AJ_b (C(\Z_p)_Q)$. That is the key behind the following proposition.

\begin{proposition}[{\cite[Theorem~4.1]{PimThesis}}]
\label{kappa}
Let $Q$ be an $\F_p$-point of $C$ such that there exists an element $T \in J(\Z_{(p)})_{Q-\overline{b}}$. The zero set $Z(\lambda_{\GC}^Q)$ is equal to $M_{Q-\overline{b}} \cap \AJ_b (C(\Z_p)_Q) = (M_0 + T) \cap \AJ_b (C(\Z_p)_Q)$.
\end{proposition} 
Thus $\lambda_\GC^Q$ consists of the equations we will use to compute Chabauty's finite set explicitly.
\begin{definition}
\label{def:GC}
Let \begin{equation} C(\Z_p)_\GC \colonequals  \bigcup_{\substack{Q \text{ s.t.} \\  J(\Z_{(p)})_{Q-\overline{b}}\neq \emptyset} }  Z(\lambda^Q_{\GC})\end{equation}
be the geometric linear Chabauty set.
\end{definition}

In practice, the $\lambda_i$ can only be calculated in finite $p$-adic precision, where, because they are given by convergent power series, they become polynomials. Although one can say quite a lot about the degrees of these polynomials \cite[Lemma~3.7]{PimThesis}, this method is especially fruitful modulo $p$, where the $\lambda_i$ become affine linear polynomials. To give an upper bound on $Z(\lambda_\GC^Q)$, one can use the following theorem.

\begin{proposition}[{\cite[Theorem~4.12]{EdixhovenLido}}]
\label{finiteness}

Let $A = \Z_p\angle{x_1,\dots,x_{r'}}/(\lambda_1,\dots,\lambda_{g-1})$ and $\bar{A} \colonequals \F_p[ x_1, \dots, x_{r'} ]/(\lambda_1, \dots, \lambda_{g-1})$ its reduction modulo $p$. Assume $\bar{A}$ is finite. Then $\bar{A}$ is Artinian and so $\bar{A}\simeq \prod_{m \in \mathrm{MaxSpec}(\bar{A})} \bar{A}_{m}$. 

We have the following upper bound on $|\Hom_{\Z_p}(A,\Z_p)|$ and hence on the number of points in $C(\Z_{(p)})_Q$: \[\sum_{m} \dim_{\F_p} \bar{A}_{m} \geq |\Hom_{\Z_p}(A,\Z_p)| \geq C(\Z_{(p)})_Q\] where the sum is taken over $m$ such that $\bar{A}/m \bar{A} = \F_p$.

\end{proposition}

By Proposition \ref{finiteness}, as long as $\bar{A}$ is finite-dimensional, it suffices to compute $\lambda_i$ modulo $p$ to obtain upper bounds for $C(\Z_{(p)})$. 
The $\lambda_i$ are affine linear modulo $p$, so $\bar{A}$ can only be finite-dimensional if it is $\F_p$ or the zero ring, which happens if the linear system of equations $\{ \lambda_i \equiv 0 \bmod p \text{ for all } i\}$ has respectively one or zero solution(s). This observation enables the following reformulation of Proposition~\ref{finiteness}.

\begin{corollary}
\label{cor-linalg-jzp}
Assume $M_0/pM_0 \to J(\Z/p^2\Z)_0$ is injective with image $\Mbar_0$. If in every residue disk of $J(\Z/p^2\Z)$ there is at most one intersection between the image $\Mbar$ of $J(\Z_{(p)})$ and $\AJ_b (C(\Z/p^2\Z))$, then $|C(\Z_{(p)})| \leq |\Mbar \cap \AJ_b (C(\Z/p^2\Z))| \leq |C(\F_p)|$. 
\end{corollary}

In particular, if the (not necessarily homogeneous) linear system of equations $\{ \lambda_i \equiv 0 \bmod p \text{ for all }i \}$ in Proposition \ref{finiteness} has zero or one solution, there is respectively zero or at most one point in $C(\Z_{(p)})_Q$.

\begin{definition}
\label{def:goodreduction}
We say that the Mordell--Weil group is \defi{of good reduction} (modulo $p$) if the map $M_{0}/pM_0 \to J(\Z/p^2\Z)_{0}$ is injective. Otherwise, we say that it is \defi{of bad reduction}.
\end{definition}

\begin{remark}
\label{rem-basepoint-indep}
This method is independent of the choice of $b \in C(\Z_{(p)})$. Choosing a different basepoint $b'$ shifts the image of the curve by $b-b'$. Equivalently, it shifts $\Mbar$ by $b'-b$. But $b'-b$ is an element of the Mordell--Weil group, so the translation is equal to $\Mbar$.
\end{remark}

\begin{remark}
\label{rem-generators-for-M0}
In genus $g>2$, finding the generators of the Mordell--Weil group can be intractable using the current methods. Often, one can hazard a guess by giving a subgroup $G \subset M$ (for example, taking the subgroup generated by the differences of rational points on $C$ of bounded height). But verifying that the given subgroup is indeed the entire Mordell--Weil group is a very difficult task. For genus $3$ hyperelliptic curves, this can be done (see \cite{StollHeights}), but it may take weeks of CPU time. However, to execute the geometric linear Chabauty algorithm in a fixed residue disk over $Q \in C(\F_p)$, given $T \in J(\Z_{(p)})_{Q- \overline{b}}$, we do not need generators of the Mordell--Weil group; we only need $p$-adic generators of the closure of the kernel of reduction of the Mordell--Weil group $\overline{M_0}$. This is immediately satisfied if the index $[M : G]$ is not divisible by $p$, equivalently, if $G$ is saturated at $p$. The condition that $G$ is saturated at $p$ can be checked by reducing $G$ modulo $\ell$ for small primes $\ell$. Regarding the required $T \in J(\Z_{(p)})_{Q- \overline{b}}$, we can just produce such a $T$; if instead we want to prove that it does not exist, we need that the image of $G$ in $J(\F_p)$ is equal to the image of $M$; it is enough to check $G$ is saturated at (some of) the primes dividing $|J(\F_p)|$.

For applications and more details, see Examples \ref{ex-g3-MW} and \ref{ex-g3-badred}.
\end{remark}

\subsection{The modulo \texorpdfstring{$p$}{p} method}
\label{themodulopmethod}
We now describe how to translate geometric linear Chabauty modulo $p$ into $\F_p$-linear algebra. For each $Q \in C(\F_p)$, we can find $T \in J(\Z_{(p)})_{Q- \overline{b}}$, or there is no rational point in the residue disk $C(\Z_p)_Q$ (these two options are not mutually exclusive); this is explained in greater detail in Section \ref{subs:sieve}. Fix a choice of $T$. We want to calculate $\AJ_b (C(\Z_p)_Q) \cap M_{Q-\overline{b}}$ by finding the affine linear polynomials $\lambda_i \mod p$ of Proposition~\ref{kappa}. To calculate these linear polynomials modulo $p$ it suffices to work in residue disks of $J(\Z/p^2\Z)$. In this section, we assume the Mordell--Weil group is of good reduction.

Choosing a parameter $t_Q$ for $C$ at $Q$ gives a bijection $\widetilde{t}_Q \colon C(\Z/p^2 \Z)_Q \xrightarrow{\sim} \F_p$; we write $Q_\mu$ for the point mapping to $\mu $. In the same way, by choosing parameters, we have an isomorphism $J(\Z/p^2\Z)_0 \simeq \F_p^g$ as groups. 
After translation by $-T$, we see $C(\Z/p^2\Z)_Q$ embeds as a $1$-dimensional affine subspace of $J(\Z/p^2\Z)_0$ by the map $C(\Z/p^2\Z)_Q \to J(\Z/p^2\Z)_0,$ sending $x \mapsto x-b-T$. 

Write $M_{Q-\overline{b}} = T + M_0$. Let $\Mbar$ denote the image of $M$ in $J(\Z/p^2\Z)$; then $\Mbar_0 \cong \F_p^{r'}$ and we see that $(\AJ_b (C(\Z/p^2\Z)_Q ) \cap \Mbar_{Q-\overline{b}}) - T$ is exactly $(\AJ_b (C(\Z/p^2\Z)_{Q}) - T) \cap \Mbar_0$.

Now, let $D_Q \subset J(\Z/p^2\Z)_0$ be the one-dimensional subspace 
\begin{equation}
D_Q \colonequals \{Q_\mu - Q_0 : \mu \in \F_p\},
\end{equation} 
and let $v\colonequals Q_0 - b - T$. We can rephrase Corollary \ref{cor-linalg-jzp} purely in terms of linear algebra: let $\phi$ denote the linear map $\phi: D_Q \oplus \Mbar_0 \to J(\Z/p^2\Z)_0$ arising from taking the sum of the embeddings $D_Q, \Mbar_0 \subset J(\Z/p^2\Z)_0$, then by the equations $\AJ_b (C(\Z/p^2\Z)_Q) = D_Q + Q_0 - b$ and $\Mbar_{Q-\overline{b}} = \Mbar_0 + T$ we get
\begin{equation}
\label{modpGC}
|(\AJ_b (C(\Z/p^2\Z)_Q) \cap \Mbar_{Q-\overline{b}})| = |\phi^{-1}(v)|.
\end{equation}

\begin{remark}
If we know there is a rational point $P$ in the residue disk $C(\Z_p)_Q$, then we can take $t_Q$ to be a parameter at $P$, and choose $T = P - b$ to get $v = 0$ and hence $|\phi^{-1}(0)|$ on the right side of this equation. (In general, $|\phi^{-1}(0)|$ is always an upper bound on $|\phi^{-1}(v)|$.)
\end{remark}

\subsection{The Mordell--Weil sieve}
\label{subs:sieve}
For each residue disk $C(\Z_p)_Q$, if it contains a rational point then there exists $T \in J(\Z_{(p)})_{Q-\overline{b}}$. Assuming we have generators of a subgroup of the Mordell--Weil group that has the same image in $J(\F_p)$, as described in Remark \ref{rem-generators-for-M0}), the existence of $T$ can be checked by a simple calculation in $J(\F_p)$. This calculation can be thought of as a Mordell--Weil sieve \cite{BruinStollMW,SiksekMordell} at the single prime $p$. The Mordell--Weil sieve is a more general technique that produces information about congruence conditions of rational points for subvarieties of abelian varieties.

To determine whether $T$ exists, we consider the diagram
\begin{equation}
\label{sieve}
\begin{tikzcd}
C(\Z_{(p)}) \arrow[r, hook, "\AJ_b"] \arrow[d] & J(\Z_{(p)}) \arrow[d, "\alpha"] \\
C(\F_p) \arrow[r, hook, "\beta"]         & J(\F_p)        
\end{tikzcd}.
\end{equation}
For $Q \in C(\F_p)$, if $\beta(Q)$ is not in the image of $\alpha$, then we say $Q$ \defi{fails the Mordell--Weil sieve} (at $p$). In this case, the residue disk $C(\Z_p)_Q$ cannot contain a rational point. Otherwise, $Q$ \defi{passes the Mordell--Weil sieve} (at $p$).

\subsection{The Chabauty--Coleman method}
\label{S:ChabCol}

We briefly outline the Chabauty--Coleman method for producing a finite set of $p$-adic points $C(\Z_p)_{\CC} \subset C(\Z_p)$ that contains the rational points $C(\Z_{(p)})$. For more details and other perspectives on the method, we refer the reader to \cite{Torsion,Wetherell,McCallumPoonen}.

Fix a basepoint $b \in C(\Z_{(p)})$ and consider the inclusion of the curve into the Jacobian $\AJ_{b}\colon C(\Z_p) \to J(\Z_p)$ via the Abel--Jacobi map. Coleman \cite[Theorem~2.11]{Torsion} defined a $p$-adic integral on the curve $C$. The Coleman integral on regular one-forms agrees with the logarithm on $J(\Z_p)_0$ interpreted as a $p$-adic Lie group via the equality $J(\Q_p)_0 = J(\Z_p)_0$. We recall some properties of the logarithm here.

\begin{remark}
Much of the literature on formal groups and the logarithm works with $\Q_p$-vector spaces. We will need results about $\Z_p$-modules. Most results carry over; for details on the $\Z_p$-module case we reference \cite[Section~3]{PimThesis}; for the $\Q_p$-vector space case see \cite{honda70} or \cite[III \S7]{Bourbaki}.
\end{remark}

Recall $H^0(J_{\Z_p},\Omega_{J_{\Z_p}}^1)$ is a free $\Z_p$-module of rank $g$. For any element $j \in J(\Z_p)$, we have an element 
\begin{equation}
\label{firstlogdef}
\log(j) \colonequals \frac{1}{p} \int_{0}^{j} \in \Hom_{\Z_p}(H^0(J_{\Z_p},\Omega_{J_{\Z_p}}^1),\Q_p),
\end{equation}
sending a differential $\omega$ to the logarithm $1/p\int_0^j \omega$. The resulting map $\log\colon j \mapsto 1/p \int_0^j$ is a homomorphism of abelian groups.

\begin{remark}
The value of the logarithm in (\ref{firstlogdef}) is defined to be $1/p$ the value of the usual Coleman integral. We divide by $p$ to renormalize: the value $\int_0^{j}\omega$ is always divisible by $p$ if $j \in J(\Z_p)_0$.
\end{remark}

\begin{proposition}[{\cite[Lemma~3.7]{PimThesis}}]
\label{codomainlog}
Recall that we assume $p > 2$; then the logarithm induces an isomorphism of abelian groups on the kernel of reduction $J(\Z_p)_0 \xrightarrow{\sim} H^0(J_{\Z_p},\Omega_{J_{\Z_p}}^1)^\vee$, where the dual is taken in the category of $\Z_p$-modules.

Write $m \colonequals \Ann J(\F_p)$ to denote the smallest positive integer such that $m \cdot j = 0$ for all $j \in J(\F_p)$. In particular, the integral $1/p\int_0^j \colonequals 1/p \cdot ( 1/m \cdot \int_0^{mj}) $ lands in the submodule $\Hom_{\Z_p}(H^0(J_{\Z_p},\Omega_{J_{\Z_p}}^1),(1/\Ann J(\F_p)) \cdot \Z_p)$.
\end{proposition}

As $\Omega_{J_{\Z_p}}^1$ is locally free,
\begin{align*}\log: J_{\Z/p^n\Z}(\Z/p^n\Z)_0 &\to H^0(J_{\Z/p^n\Z},\Omega^1_{J_{\Z/p^n\Z}})^\vee \tensor \Z/p^{n-1}\Z, \\ 
j &\mapsto 1/p \int_0^j 
\end{align*} is an isomorphism given by lifting $j$ to $\Z_p$, taking $\log$, then reducing modulo $p^{n-1}$. If $|J(\F_p)|$ is invertible in $\Z_p$, this even extends to a morphism 
\begin{equation}
J_{\Z/p^n\Z}(\Z/p^n\Z) \to H^0(J_{\Z/p^n\Z},\Omega^1_{J_{\Z/p^n\Z}})^\vee \tensor \Z/p^{n-1}\Z.
\end{equation} 
Choosing a basis $(\omega_i)_{i = 1}^g$ of $H^0(J_{\Z_p},\Omega_{J_{\Z_p}}^1)$ and dualizing, we get $\log: J(\Z_p)_0 \to \Z_p^g$, reducing to $\log: J(\Z/p^n\Z)_0 \to (\Z/p^{n-1}\Z)^g.$

Fix a point $R \in J(\F_p)$.
For $j \in J(\Z_p)_R$, the logarithm $\log(j)$ has a convergent power series expansion \cite[Lemma~3.7]{PimThesis}. Let $t_1, \dots, t_g$ be local parameters of $J$ at $R$ and expand $\omega_i(t_1, \dots, t_g) = \sum_{i = 1}^g f_i(t_1, \dots, t_g) d t_i$, with $f_i \in \Z_p [[t_1, \dots, t_g]]$.

By formally integrating, $\omega_i$ has a unique local antiderivative $g_i$ on $J(\Z_p)_R$ such that $d g_i = \omega_i$ and $g_i \in \Q_p[[t_1, \dots, t_g]]$ with constant term $0$. Let $\tilde{R} \in J(\Z_p)_R$ be the point where all $t_i$ vanish. We may then evaluate the power series at $j$ using the local parameters at $R$ by 
\begin{equation}
\label{logpowerseries}
 \log(j) \colonequals (g_1(t_1(j), \dots, t_g(j))/p , \dots, g_j(t_1(j), \dots, t_g(j))/p) + \log(\tilde{R})
\end{equation}

\begin{remark}
For computational purposes, it is easier to exploit the isomorphism $\AJ_b^*:H^0(C, \Omega_{C}^1) \simeq H^0(J, \Omega_{J}^1)$. Then we may evaluate $\log (j)$ using linearity of the logarithm and expanding in a local parameter on $C_{\Z_p}$ at each point. As $C$ is one-dimensional over $\Z_{(p)}$, we only need one parameter, see \cite{NATOLectures} for example.
\end{remark}

Consider the inclusion of $M$ into $J(\Z_p)$. Let 
\begin{equation}
\label{vanishingdiff}
V \colonequals \{ v \in H^0(J_{\Z_p}, \Omega_{J_{\Z_p}}^1) : 1/p \int_{0}^m v = 0 \text{ for all } m \in M\}.
\end{equation}
Since $\rank_{\Z_p}  H^0(J_{\Z_p}, \Omega_{J_{\Z_p}}^1)  =g$   but $\rank_{\Z_p} M = r'$, we see $V$ is a rank $(g-r')$ $\Z_p$-module. Let $B$ be a basis for $V$ and let $v \colon \Hom_{\Z_p}(H^0(J_{\Z_p},\Omega_{J_{\Z_p}}^1),(1/\Ann J(\F_p)) \cdot \Z_p) \to \frac{1}{\Ann J(\F_p)} \Z_p^{g-r'}$ denote the map $\psi \mapsto (\psi(\nu))_{\nu \in B}$. By construction, the map $v$ vanishes on $\log(j)$ for $j \in M$.

Next consider the Abel--Jacobi embedding $\AJ_b\colon C(\Z_p) \to J(\Z_p)$ and the composition $\Lambda_{\CC} \colonequals v\circ \log \circ \AJ_b $. We get the following diagram:

\begin{equation}\label{eqn:CC}\begin{tikzcd}
M& J(\Z_p) & {\Hom_{\Z_p}(H^0(J_{\Z_p},\Omega_{J_{\Z_p}}^1),\frac{1}{\Ann J(\F_p)} \cdot \Z_p)} & {(\frac{1}{\Ann J(\F_p)})\cdot \Z_p^{g-r'}} \\
  & C(\Z_p) \\
  & {0}
  \arrow["{\AJ_b}", from=2-2, to=1-2]
  \arrow[from=1-1, to=1-2]
  \arrow["{\log}", from=1-2, to=1-3]
  \arrow[from=3-2, to=2-2]
  \arrow["{v}", from=1-3, to=1-4]
  \arrow["{\Lambda_{\CC}}", from=2-2, to=1-4]
\end{tikzcd}
\end{equation}

\begin{definition}
\label{def:CC}
We define the Chabauty--Coleman set to be the following:
\begin{equation}
\label{CCdef}
C(\Z_p)_{\CC} \colonequals Z(\Lambda_{\CC})  = \{ P \in C(\Z_p) : 1/p \int_0^{P-b} \nu  = 0 \text{ for all } \nu \in B\}.
\end{equation}
\end{definition}
Then $\AJ_b(C(\Z_{p})) \cap M$ is contained in $C(\Z_p)_{\CC}$.

\begin{lemma}
Fix $R = \AJ_b (Q) \in J(\F_p)$ with $Q\in C(\F_p)$. Let $\widetilde{z} : C(\Z_p)_Q \xrightarrow{\sim} \Z_p$ be given by a parameter $z$ at $Q$, and denote by $Q_{\mu}$ the element with image $\mu$. Then $\log(Q_{\mu}-b)$ can be expressed as $f+ c$ where $f \in \Z_p\angle{\mu}^g$ and $c \in 1/\Ann(\F_p) \cdot  \Z_p^g$.
\end{lemma}

\begin{proof}
Write $\log(Q_{\mu}- b) = \log(Q_{\mu} - Q_0) + \log(Q_0 - b)$, and $t_1,\dots,t_g$ for parameters of $J$ at $0$. Then by \cite[Remark~2.3]{PimThesis} the function $\mu \mapsto \widetilde{t_i}(Q_{\mu} - Q_0)$ is given by a convergent power series in $\mu$. Also, $\log \colon J(\Z_p)_0 \to \Z_p^g$ consists of $g$ convergent power series in $\widetilde{t_1},\dots,\widetilde{t_g}$. The composition of convergent power series exists and is itself a convergent power series, so $\log(Q_\mu - Q_0) \in \Z_p\angle{\mu}^g$.

By Proposition~\ref{codomainlog}, the constant term $\log(Q_0 - b)$ lives in $1/\Ann(\F_p) \cdot  \Z_p^d$.
\end{proof}

In practice, to compute $C(\Z_p)_{\CC}$, we must truncate the power series by working modulo $p^n$ for some $n \in \Z_{>0}$. The choice of $n$ depends on the Newton polygon of the power series: $n$ must be large enough so that the truncated power series has the same number of zeros as the original power series, allowing us to Hensel lift the solutions of the truncated power series to $\Z_p$.

To compare the geometric linear Chabauty and Chabauty--Coleman methods, we describe how the Chabauty--Coleman method works in a single residue disk. We fix an $\F_p$-point $Q \in C(\F_p)$, and assume $Q$ passes the Mordell--Weil sieve, i.e. $J(\Z_{(p)})_{Q-\overline{b}}$ contains an element $T$. Since $T \in M$, we know $1/p \int_b^T v$ vanishes, so $\ker (v \circ \log) = \ker (v \circ \log \circ \tr_{-T})$. Then diagram (\ref{eqn:CC}), restricted to this residue disk, becomes
\begin{equation}\label{eqn:CCQ}\begin{tikzcd}
M_{Q-\overline{b}}& J(\Z_p)_{Q-\overline{b}} & H^0(J_{\Z_p},\Omega_{J_{\Z_p}}^1)^\vee & \Z_p^{g-r'} \\
  & C(\Z_p)_Q \\
  & {0}
  \arrow["{\AJ_b}", from=2-2, to=1-2]
  \arrow[from=1-1, to=1-2]
  \arrow["{\log \circ \tr_{-T}}", from=1-2, to=1-3]
  \arrow[from=3-2, to=2-2]
  \arrow["{v}", from=1-3, to=1-4]
  \arrow["{\lambda_{\CC}^Q}", from=2-2, to=1-4]
\end{tikzcd}
\end{equation}
where $\lambda_{\CC}^Q$ is now the composition $v \circ \log \circ \tr_{-T} \circ \AJ_b$. Note that $\log \circ \tr_{-T}$ is a bijection $J(\Z_p)_{Q-\overline{b}} \to H^0(J_{\Z_p},\Omega_{J_{\Z_p}}^1)^\vee$.

Unfortunately, the sequence \[0 \to M_{Q-\overline{b}} \xrightarrow{\log \circ \tr_{-T} \circ \kappa} H^0(J_{\Z_p},\Omega^1_{J_{\Z_p}})^{\vee} \xrightarrow{v} \Z_p^{g-r'} \to 0\] is not necessarily exact at the middle term. In fact, $\ker (v \circ \log)$ is the $p$-saturation $N_0$ of $M_0$ inside $J(\Z_p)_0$, by the following lemma.
\begin{lemma}
\label{zpmodules}
Let $A$ be a free $\Z_p$-module of rank $n$, let $B$ be a $\Z_p$-submodule of rank $m$, and let $v\colon A \to \Z_p^{n-m}$ be a full rank linear map vanishing on $B$. Then $\ker v$ is the $p$-saturation of $B$.
\end{lemma}
\begin{proof}
By linearity of $v$,  $\ker v$ contains the $p$-saturation of $B$. Comparing dimensions, we see that $(\ker v) \tensor_{\Z_p} \Q_p$ must equal $B \tensor_{\Z_p} \Q_p$. Then $\ker v$ is contained in $(B \tensor_{\Z_p} \Q_p) \cap A$, which is exactly the $p$-saturation of $B$.
\end{proof}

Applying Lemma \ref{zpmodules} to  $A= J(\Z_p)_0$ and $B = M_0$,  we see $\ker (v \circ \log \circ \tr_{-T}) = T + N_0$. That gives us the following corollary.

\begin{corollary}
\label{psat}
Let $Q$ be an $\F_p$-point of $C$ that passes the Mordell--Weil sieve, with $T \in J(\Z_{(p)})_{Q-\overline{b}}$. Then $Z(\lambda_\CC^Q)$ is exactly the intersection $(N_0 + T) \cap \AJ_b(C(\Z_p)_{Q-\overline{b}})$ pulled back along $\AJ_b$.
\end{corollary}

\section{Explicit Geometric linear Chabauty mod \texorpdfstring{$p$}{p}}
\label{S:ExplicitModP}

We outline a practical method for doing explicit geometric linear Chabauty modulo $p$ using Coleman integration. Previously, in \cite{PimThesis}, this was done by using the birationality of the map $\Sym^g C \to J$ given by subtracting a generic degree $g$ divisor, and using the Khuri-Makdisi representation of elements of the Jacobian (\cite{Makdisi04}), where elements of the Jacobian are represented as certain submodules of Riemann--Roch spaces. This approach of using the birationality of the map $\Sym^g C \to J$ is taken in \cite{EdixhovenLido} for geometric quadratic Chabauty as well. 

The advantage of using Coleman integration is that the map $J(\Z/p^2\Z)_{0} \to \F_p^g$ can be made much more explicit, making the computations simple linear algebra. In what follows, we describe this map and give examples of the method.

The logarithm is linear modulo $p$ for $p > 2$ \cite[Lemma~3.7]{PimThesis} (when $p = 2$ the logarithm is not necessarily linear modulo $p$, hence we exclude this case). We choose parameters for $J(\Z/p^2\Z)_0$ and a $\Z_p$-basis of $H^0(J_{\Z_p},\Omega_{J_{\Z_p}}^1)$. Then again by \cite[Lemma~3.7]{PimThesis} the reduction modulo $p$ of the logarithm, i.e. the map $\log\colon \F_p^g \to \F_p^g$, is an isomorphism of vector spaces over $\F_p$, allowing us to carry out the methods in Section \ref{GeoChab}. For example, the vector in $\F_p^g$ corresponding to $j \in J(\Z/p^2\Z)_{0}$ is $(1/p \int_0^j \omega_i)_{i=1}^g$. This translates linear algebra in a vector space $J(\Z/p^2\Z)_0$ of dimension $g$ where addition is difficult, to linear algebra in $\F_p^g$. This is the final step needed to perform geometric linear Chabauty modulo $p$.

\subsection{Examples}
\begin{example}
\label{ex-g2-first}
Let $C/\Z_{(5)}$ be (the smooth projective model of) the genus $2$ curve (LMFDB label \href{https://www.lmfdb.org/Genus2Curve/Q/10989/a/10989/1}{\texttt{10989.a.10989.1}}) given by \[ y^2 = x^5 + x^3 + x^2 + 1/4\] with Mordell--Weil rank 1. Then $C$ has the known rational points 
\[C(\Z_{(5)})_{\text{known}} = \{\infty = (1 : 0 : 0),P_1 = (0 : -1/2 : 1) ,P_2 = (0 : 1/2 : 1)\} .\]
The Mordell--Weil group of $J$ is isomorphic to $\Z$ and generated by $P_1 - \infty$.

Let $p = 5$. Over $\F_5$ we have the points \[C(\F_5)=  \{ (1 : 0 : 0), (0 : 2 : 1), (0 : 3 : 1) \}.\] 

At the finite non-Weierstrass residue disks corresponding to the points $\overline{P}_1=(0 : 2 : 1)$ and $\overline{P}_2=(0 : 3 : 1)$, we have the local parameter $x$ giving isomorphisms $C(\Z/5^2\Z)_{\overline{P}_i} \xrightarrow{\sim} 5 \Z/5^2\Z$. At the infinite point, the local parameter is $t= x^2/ y$. We identify $J(\Z/5^2\Z)_0$ with $\F_5^2$ by choosing the basis of differentials $\omega_0 = d x/y , \omega_1 = x d x/y$, then applying $\log\colon j \mapsto (1/5 \int_0^j \omega_0,1/5 \int_0^j \omega_1)$. 

Consider the residue disk $C(\Z_5)_{\overline{P}_1}$: our goal is to show there is only one point in this disk (and each other disk). We start by computing $\Mbar_0$. Since $P_1 - \infty$ generates the Mordell--Weil group, computing $\Mbar_0$ is equivalent to finding the smallest $n$ such that $n(P_1 - \infty) = 0$ in $J(\F_5)$. We find $n = 15$, that is, $m = 15 (P_1 - \infty)$ generates $M_0$. A simple calculation with tiny Coleman integrals shows that $\log m = (3,1) \in \F_5^2$. That automatically means that the map $M_0 \to J(\Z_5)_0$ is of good reduction, and $\Mbar_0$ is the $\F_5$-vector space generated by $(3,1)$.

By specializing $\lambda = 0$ and $\lambda= 1$ we see that $D_{\overline{P}_1}  \subset J(\Z/5^2\Z)_0$  is generated by $d = \\(5:-1/2:1) -  (0:-1/2:1)$ and $\log d = (4,0)$.

Now the matrix $A$ representing the map $\phi\colon D_{\overline{P}_1} \oplus \Mbar_0 \to J(\Z/5^2\Z)_0$ is \[\begin{pmatrix}4 & 3 \\ 0 & 1\end{pmatrix}.\] As $A$ is invertible, $|\phi^{-1}(v)| = 1$. Hence, by (\ref{modpGC}) the only rational point in the residue disk of $P_1$ is $P_1$ itself. Using the hyperelliptic involution, we see the same holds for $P_2$.

For the point $\infty$, we carry out a similar calculation. We change our basepoint to $\infty$, allowing us to again work in 
$J(\Z/5^2\Z)_0 = \F_5^2$. Then $D_{\overline{\infty}}$ is generated by $d' = (1:5:0) - \\(1:0:0)$ and $\log d' = (0,4)$. So we again conclude that the only rational point in the residue disk containing $\infty$ is $\infty$.

As we have now treated all three residue disks, we have proven \[C_\Q(\Q) = \{(1 : 0 : 0),  (0 : -1/2 : 1),  (0 : 1/2 : 1)\}.\]
\end{example}

\begin{example}
\label{ex-g2-rule-out}
Let $C/\Z_{(3)}$ be (the smooth projective model of) the genus $2$ curve (LMFDB label \href{https://www.lmfdb.org/Genus2Curve/Q/29395/a/29395/1}{\texttt{29395.a.29395.1}}) given by the equation
\[y^2 + (x^2 + x + 1)y = x^5 - x^4 + x^3\]
with Mordell--Weil rank 1. Then $C$ has known rational points \[C(\Z_{(3)})_{\text{known}} = \{\infty = (1 : 0 : 0),(0 : 0 : 1),(0 : -1 : 1)\} .\] The Mordell--Weil group is $J(\Z_{(p)}) \simeq \Z$ and is generated by $d \colonequals (0 : -1 : 1) - \infty$.

Let $p = 3$, then \[C(\F_3)=  \{ (1 : 0 : 0), (0 : 0 : 1), (0 : 2 : 1), (1 : 1 : 1), (1 : 2 : 1), (2 : 0 : 1), (2 : 2 : 1)\}.\] In this example we show how to rule out some of the residue disks that do not contain rational points using geometric linear Chabauty.

For each residue disk $C(\Z_3)_Q$ not containing a known rational point, using arithmetic in $J(\F_3)$ we are able to find $T \colonequals  m d$ such that $T  \in J(\Z_{(3)})_{Q - \overline{\infty}}$.  The order of $\overline{d}$ in $J(\F_3)$ is $29$ so $m <29$. Since all $Q \in C(\F_3)$ pass the Mordell--Weil sieve, we proceed to compute the matrix $A$ for each $Q$.

First we compute $\Mbar_0$, which does not depend on $Q$. To do this, we compute $\log (29d) =  (2,2) \in \F_3^2$.

Then, for each residue disk $C(\Z_3)_Q$ without a known rational point, we will compute the one-dimensional subspace $D_Q$. To do this, we lift $Q$ in two different ways $Q_1$ and $Q_2$ to finite precision, and then take the tiny Coleman integral $\log (Q_1 - Q_2)$.
\begin{table}[h!]
\begin{tabularx}{452pt}{lllll}
$Q$ & $m$ & $Q_1$& $Q_2$ &$\log(Q_1 - Q_2)$  \\
\midrule
$(1 : 1 : 1)$ & $20$& $( 1 + O(3^3): 4 + O(3^3):1)$& $(4+ O(3^3):1+ O(3^3):1)$ & $(0, 2)$ \\
$(1 : 2 : 1)$& $9$  & $(1+O(3^3): 20+O(3^3):1)$&$(4+O(3^3): 5 +O(3^3):1)$ & $(0, 1)$ \\
$(2 : 0 : 1)$& $16$ & $(2+O(3^3): 6+O(3^3):1)$&$( 5+O(3^3):6 +O(3^3):1)$ &$(2, 2)$\\
$(2 : 2 : 1)$& $13$ & $(2+O(3^3):14+O(3^3):1)$&$( 5+O(3^3): 17 +O(3^3):1)$ &$(1, 1)$
\end{tabularx}
\caption{Values for $D_Q$} 
\end{table}

For the $\F_3$-points, $(2 : 0 : 1)$ and $(2 : 2 : 1)$, we see modulo $3$ that $\Mbar_0$ and $D_Q$ give determinant zero matrices: 
\begin{equation*}
 A_{(2 : 0 : 1)} =  \begin{pmatrix} 2 & 2 \\ 2& 2 \end{pmatrix}, A_{(2 : 2 : 1)} =  
 \begin{pmatrix} 2 & 1 \\ 2& 1 \end{pmatrix}.
 \end{equation*} 
 We check whether for $v = Q_1 - \infty -  md$, the vector $\log v $ is in the image of $A_Q$. For $(2 : 0 : 1)$ we have $\log v = ( -3 + O(3^2), -2 + O(3^2) )$ and for $(2:2:1)$ we have $\log v = ( O(3^2), 2 + O(3^2) )$. Therefore by (\ref{modpGC}) neither residue disk can contain a $\Z_{(3)}$-point.

However, for $(1 : 1 : 1)$ and $(1 : 2 : 1)$, we see that $A_Q$ is invertible: 
\begin{equation*}
 A_{(1 : 1 : 1)} =  \begin{pmatrix} 2 & 0 \\ 2& 2 \end{pmatrix}, A_{(2 : 2 : 1)} =  \begin{pmatrix} 2 & 0 \\ 2& 1 \end{pmatrix}. 
 \end{equation*}
Hence geometric linear Chabauty modulo $3$ shows that there is at most one rational point in each of the two corresponding residue disks.

It is possible to show that there are no rational points in these residue disks, for example using a Mordell--Weil sieve at strategically chosen primes $\ell \neq 3$.
\end{example}
In the previous pair of examples, the geometric linear Chabauty method and Chabauty--Coleman method find the same set of $3$-adic points, i.e. $C(\Z_3)_\CC = C(\Z_3)_\GC$. In the following example we show this is not always the case. We will study the differences between the two methods in the final section. 
\begin{example}
\label{ex-g2-different}
Let $C/\Z_{(3)}$ be (the smooth projective model of) the genus $2$ curve (LMFDB label \href{https://www.lmfdb.org/Genus2Curve/Q/9470/a/37880/1}{\texttt{9470.a.37880.1}}) given by the equation \[y^2 + xy = x^5 + 2x^4 + 4x^3 + 4x^2 + 3x + 1\]
with Mordell--Weil rank 1. Then $C$ has the known rational points \[C(\Z_{(3)})_{\text{known}} = \{ (1 : 0 : 0), (0 : -1 : 1), (0 : 1 : 1)\}.\] But 
\[C(\F_3) = \{(1 : 0 : 0), (0 : 1 : 1), (0 : 2 : 1), (1 : 0 : 1), (1 : 2 : 1), (2 : 2 : 1)\} .\]

The Mordell--Weil group of $J$ is isomorphic to $\Z$, generated by $d \coloneqq (0 : -1 : 1) - (1:0:0)$. In $J(\F_3)$, $d$ has order $11$.
Sieving, we find the only residues $c \in C(\F_3)$ such that there exists $m \in \Z$ such that $c -\infty = m d$ are the images of rational points under the reduction map.

In their corresponding residue disks, the geometric linear Chabauty method only finds one solution, so $C(\Z_3)_\GC = C(\Z_{(3)})_{\text{known}} = C(\Z_{(3)})$.

However, the Chabauty--Coleman method finds the rational points along with the $p$-adic points 
\begin{align*}
&\{ (2 + 3 + 3^2 + 2\cdot3^3 + 2\cdot3^4 + 3^5 + 3^6 + O (3^7): 2 + 2\cdot3^2 + 3^3 + 3^4 + 2\cdot3^6+ O (3^7):1), \\
&(1 + 2\cdot3^2 + 2\cdot3^3 + 2\cdot3^4 + 3^6+ O (3^7): 2\cdot3 + 3^3 + 2\cdot3^5+ O (3^7):1 ),\\ &(1 + 2\cdot3^2 + 2\cdot3^3 + 2\cdot3^4 + 3^6+ O (3^7):2 + 2\cdot3^3 + 2\cdot3^4 + 2\cdot3^5+ O (3^7):1)\}.
\end{align*}
One of the $3$-adic points lies in a Weierstrass disk. Since it is the only point in its disk and is fixed by the hyperelliptic involution, the point is $2$-torsion, while the other two points do not readily have explanations for being in the Chabauty--Coleman set (in particular, they are not torsion in $J(\Z_3)$ and not recognizably algebraic).
\end{example}

\section{Comparison}
\label{comparison}
Throughout this section, $Q$ still denotes a point in $C(\F_p)$ and $T$ still denotes a point in $J(\Z_{(p)})_{Q-\overline{b}}$ (i.e., we assume $Q$ passes the Mordell--Weil sieve at $p$, see Section \ref{subs:sieve}).

To compare the geometric linear Chabauty and Chabauty--Coleman methods, we first recall some notation in the following commutative diagram, which is the union of the diagrams (\ref{diag:GC}) and (\ref{eqn:CCQ}):
\begin{equation}\label{eqn:comp}\begin{tikzcd}
  && {0} \\
  && {\Z_p^{g-1}} \\
  {0} & {M_{Q-\overline{b}}} & {J(\Z_p)_{Q-\overline{b}}} & {H^0(J_{\Z_p},\Omega^1_{J_{\Z_p}})^\vee} & {\Z_p^{g-r'}} \\
  && {C(\Z_p)_Q} \\
  && {0}
  \arrow["{\kappa}", from=3-2, to=3-3]
  \arrow["{f}", from=3-3, to=2-3]
  \arrow["{\AJ_b}", from=4-3, to=3-3]
  \arrow["{\log \circ \tr_{-T}}", from=3-3, to=3-4]
  \arrow[from=3-1, to=3-2]
  \arrow[from=2-3, to=1-3]
  \arrow[from=5-3, to=4-3]
  \arrow["{v}", from=3-4, to=3-5]
  \arrow["{\lambda_{\GC}^Q}", from=3-2, to=2-3]
  \arrow["{\lambda_{\CC}^Q}", from=4-3, to=3-5]
\end{tikzcd}
\end{equation}
where we recall
\begin{itemize}
  \item the maps $\kappa$ and $f$ are defined as in diagram (\ref{diag:GC});
  \item $\AJ_b$ is the Abel--Jacobi embedding at $b \in C(\Z_{(p)})$ from $C(\Z_p)_Q$ to $J(\Z_p)_{Q-\overline{b}}$;
  \item the map $v$ is given by $g - r'$ linearly independent Coleman integrals that vanish on $M$;
  \item and $\log : J(\Z_p)_0 \xrightarrow{\sim} H^0(J_{\Z_p},\Omega_{J_{\Z_p}}^1)^\vee$ is given by the (normalized) Coleman integral $\log\colon x \mapsto (\omega \mapsto 1/p \int_0^x \omega)$.
\end{itemize}

Now we can give a comparison theorem for the geometric linear Chabauty and Chabauty--Coleman methods.
\begin{theorem}
\label{mainthmcomp}
Let $C(\Z_p)_\GC$ and $C(\Z_p)_\CC$ be the finite subsets of $C(\Z_p)$ defined in Definitions \ref{def:GC} and \ref{def:CC}.
We have the inclusions \[ C(\Z_{(p)}) \subseteq C(\Z_p)_\GC \subseteq C(\Z_p)_\CC.\]Furthermore, for any point $R \in C(\Z_p)_\CC \setminus C(\Z_p)_\GC$, one of the following two conditions holds:
\begin{enumerate}
  \item \label{condition1} the point $\overline{R}$ fails the Mordell--Weil sieve at $p$, i.e. the image of $R-b$ in $J(\F_p)$ is not contained in the image of $M$ in $J(\F_p)$;
  \item \label{condition2} or for $T \in J(\Z_{(p)})_{\overline{R}-\overline{b}}$, the element $\log(R-b-T)$ is not in the $\Z_p$-submodule $\log M_0$ of $H^0(J_{\Z_p},\Omega^1_{J_{\Z_p}})^{\vee}$, only in its $p$-saturation $\log N_0$.
\end{enumerate} 
\end{theorem}
\begin{proof}
We may prove this disk-by-disk. Suppose $Q \in C(\F_p)$ fails the Mordell--Weil sieve. Then $M_{Q-\overline{b}} = Z(\lambda_{\GC}^Q) = \emptyset$, and so (\ref{condition1}) holds.

Otherwise, we can find $T \in J(\Z_{(p)})_{Q-\overline{b}}$. Then, by Proposition~\ref{kappa}, we know $Z(\lambda_\GC^Q) = (M_0 + T) \cap \AJ_b(C(\Z_p)_{Q-\overline{b}})$ and by Corollary~\ref{psat} we know $\AJ_b(Z(\lambda_\CC^Q)) = (N_0 + T) \cap \AJ_b(C(\Z_p)_{Q-\overline{b}})$. So we see that $R - b $ belongs to $\AJ_b(Z(\lambda_\CC^Q)) - \kappa(Z(\lambda_\GC^Q))$ if and only if $\log \circ \tr_{-T}(R-b) = \log(R-b-T)$ is in $\log N_0 \setminus \log M_0$. (As any two choices of $T$ differ by an element of $M_0$, this statement is choice-independent.)
\end{proof}

\begin{remark}
\label{rem-bad-red}
In the case of good reduction of the Mordell--Weil group, the obstruction (\ref{condition2}) cannot occur, as then by definition $M_0$ is its own $p$-saturation.
\end{remark}

\begin{corollary}
\label{cor-pnmidjfp}
If $p \nmid |J(\F_p)|$, then Theorem \ref{mainthmcomp} (\ref{condition2}) is equivalent to $\log(R-b)$ not lying in the submodule $\log M$ of $H^0(J_{\Z_p},\Omega_{J_{\Z_p}}^1)^{\vee}$. 
\end{corollary}
\begin{proof}
Recall from Proposition \ref{codomainlog} that the isomorphism $\log\colon J(\Z_p)_0 \to H^0(J_{\Z_p},\Omega_{J_{\Z_p}}^1)^\vee$ extends to a map $\log\colon J(\Z_p) \to m^{-1} H^0(J_{\Z_p},\Omega_{J_{\Z_p}}^1)^\vee$ where $m = \Ann J(\F_p)$, by sending $x$ to $\log(mx)/m$. Under the condition $p \nmid |J(\F_p)|$, we see that $m$ is a $p$-adic unit, so the logarithm extends to a map $\log\colon J(\Z_p) \to H^0(J_{\Z_p},\Omega_{J_{\Z_p}}^1)^\vee$. Hence $\log M_0$ is equal to $\log M$, as submodules of $H^0(J_{\Z_p},\Omega_{J_{\Z_p}}^1)^\vee$. As $\log T$ is an element of $\log M$, we conclude that $\log(R-b-T)$ not lying in $\log M_0$ is equivalent to $\log(R-b)$ not lying in $\log M$.
\end{proof}

Hence the geometric method is theoretically strictly better. However, depending on the curve and the level of precision needed, the geometric linear Chabauty method can be tricky to execute; the best known method for expressing $\lambda^Q_{\GC}$ as polynomials modulo some power of $p$ uses interpolation, then one has to solve multiple power series in $r$ variables. Sometimes one can use the implicit function theorem for power series \cite[Proposition~A.4.5]{Hazewinkel} to reduce to fewer variables, but in general this can be an arduous task. Hence in practice, we advise the following adjustment of the Chabauty--Coleman method.

\begin{myalgorithm}
\label{alg:improvedCC} \hfill
\begin{enumerate}
  \item Calculate $S \colonequals C(\Z_p)_{\CC}$ using the Chabauty--Coleman method.
  \item Let $(E_i)_{i = 1}^{r'} \in J(\Z_{(p)})$ be a set of topological generators of $M_0$ and $(\omega_j)_{j=1}^{g}$ a basis of $H^0(J_{\Z_p},\Omega_{J_{\Z_p}}^1)$.
  \item Calculate $\log E_i = (1/p \int_0^{E_i} \omega_j)_{j=1}^g \in \Z_p^g$ for $i = 1,\dots,r'$.
  \item For $R \in S$, remove $R$ from $S$ if it does not pass the Mordell--Weil sieve at $p$. \label{sievestep}
  \item For $R \in S$, let $T \in J(\Z_{(p)})_{\overline{R - b}}$. If $\log(R-b-T)$ is not a $\Z_p$-linear combination of $(\log E_i)_{i=1}^{r'}$, remove $R$ from $S$. \label{otherstep}
  \item Return $S$.
\end{enumerate}
\end{myalgorithm}
From the previous discussion, we have the following theorem.
\begin{theorem}
\label{thm:improvedCC}
Algorithm~\ref{alg:improvedCC} computes $\AJ_b (C(\Z_{p}))\cap M$.
\end{theorem}
\begin{proof}
This is immediate from Theorem~\ref{mainthmcomp} and Proposition~\ref{kappa}.
\end{proof}

\begin{remark}
\label{rem-magma-impl-MW}
Note that Step \ref{sievestep} in Algorithm \ref{alg:improvedCC} executes a Mordell--Weil sieve at the single prime $p$ on top of the usual Chabauty--Coleman method. The Mordell--Weil sieve is often used in combination with Chabauty--Coleman to sieve out extra $p$-adic points that are not rational. For example, the implementation of Chabauty--Coleman in Magma for genus $2$ curves based on \cite[Section~4.4]{BruinStollMW} executes the Chabauty--Coleman method to find $ C(\Z_p)_{\CC}$ and then runs a Mordell--Weil sieve at a set of primes $\{\ell_1, \dots, \ell_n\}$ to try to determine $C(\Z_{(p)}) = C_\Q(\Q)$. This is a more extensive Mordell--Weil sieve than the one used in Algorithm \ref{alg:improvedCC}. In practice, for the genus $2$ curves in Examples \ref{ex-g2-first}, \ref{ex-g2-rule-out}, and \ref{ex-g2-different}, Magma determines $C(\Z_{(p)})$ in a fraction of a second.

The points removed in Step \ref{otherstep} pass a Mordell--Weil sieve at $p$, but are ruled out by Theorem \ref{mainthmcomp} (\ref{condition2}). They may fail a Mordell--Weil sieve at some other prime $\ell \neq p$. 
\end{remark}

In Theorem \ref{mainthmcomp} there are two theoretical obstructions to the Chabauty--Coleman method calculating $\AJ_b (C(\Z_{p})) \cap M$ exactly. We give two examples that show these both occur, and where geometric linear Chabauty outperforms Chabauty--Coleman. In light of Remark~\ref{rem-magma-impl-MW}, we turn our attention to genus $3$ curves.

The following examples were computed in Magma and Sage. The code is available at the repository \cite{PimSachigit}.
\begin{example}
\label{ex-g3-MW}
Let $C/\Z_{(5)}$ be the (smooth projective model of) the genus $3$ curve given by the equation \[y^2 = 4x^7 - 12x^6 + 16x^5 - 12x^4 + 4x^3 +4x^2 - 4x + 1 \] taken from a database of genus $3$ hyperelliptic curves over $\Q$ \cite{g3database} computed using the methods in \cite{SutherlandHyperelliptic}. According to computations with the fake $2$-Selmer group done by the method \texttt{RankBounds} in Magma, the Mordell--Weil rank of $J$ is $2$. 

Computing the Chabauty--Coleman set with $p = 5$ we find
\begin{align*}
&C(\Z_5)_\CC =  \{ \infty = (1 : 0 : 0), (0 : -1 : 1), (0 : 1 : 1), (1 : -1 : 1), (1 : 1 : 1), \\
&W \colonequals (2 + 5 + 5^2 + 2\cdot5^3 + 4\cdot5^4 + 4\cdot5^5 + 3\cdot5^6  + O(5^6): O(5^6):1)\\
& R_1 \colonequals (4 + 3\cdot5 + 2\cdot5^2 + 3\cdot5^4 + 3\cdot5^5 +  O(5^6): 4 + 2\cdot5 + 4\cdot5^3 + 5^4 + 4\cdot5^5 + O(5^6): 1)\\
& R_2 \colonequals (4 + 3\cdot5 + 2\cdot5^2 + 3\cdot5^4 + 3\cdot5^5  + O(5^6):- (4 + 2\cdot5 + 4\cdot5^3 + 5^4 + 4\cdot5^5 + O(5^6)): 1)\\
& R_3 \colonequals (3 + 5 + 2\cdot5^2 + 4\cdot5^3 + 4\cdot5^4 + 4\cdot5^5 + O(5^6):3 + 5 + 5^3 + 2\cdot5^5 +  O(5^6)+:1) \\
& R_4 \colonequals (3 + 5 + 2\cdot 5^2 + 4\cdot 5^3 + 4\cdot 5^4 + 4\cdot 5^5 +  O(5^6):-(3 + 5 + 5^3 + 2\cdot5^5 +  O(5^6)):1)
\}.
 \end{align*}
 Then $W$ is a Weierstrass point, and therefore gives rise to a $2$-torsion point in $J$, but the $R_i$ are not readily recognizable. We can verify by computing Coleman integrals that they are not torsion in $J(\Z_5)$. However, the geometric linear Chabauty method rules out the residue disks of the $R_i$.

We first compute
\begin{align*}
C(\F_5) = &\{ (1 : 0 : 0), (0 : 4 : 1), (0 : 1 : 1),(1 : 4 : 1), (1 : 1 : 1), \\
&(2 : 0 : 1),(4 : 4 : 1), (4 : 1 : 1),(3 : 3 : 1), (3 : 2 : 1)  \}.
\end{align*}

We do not have computational access to generators of the Mordell--Weil group. However, this is not needed; see Remark~\ref{rem-generators-for-M0}. Instead, we can consider the set of differences of pairs of known rational points inside the Mordell--Weil group. Computing the canonical height pairing of each of the set of differences themselves, we look for the elements with the smallest canonical height (using code from \cite{StollHeights}); let $G_1 \colonequals (0:1:1) - \infty$ and $G_2 \colonequals (1:1:1) - \infty$. We can check that $G_1$ and $G_2$ are linearly independent by computing that their logarithms are linearly independent. Let $H$ be the subgroup of the Mordell--Weil group generated by $G_1$ and $G_2$. We cannot always expect $H$ is equal to the full Mordell--Weil group $J(\Z)$, but by Remark~\ref{rem-generators-for-M0} it is enough for $H$ to be saturated at $5$ and the primes dividing $|J(\F_5)| = 340$.

To check whether $H$ is saturated at a given prime $\ell$, we compute the kernel of the map
\[
G/\ell G \to \prod_{q} J(\F_q)/\ell J(\F_q)
\]
where $q$ runs over some small primes such that $\ell \mid |J(\F_q)|$ and check if this kernel is trivial \cite[Section~12]{StollHeights} (using code from \cite{SatCheck}).  If the kernel is not trivial, then we cannot apply geometric linear Chabauty, but if we suspect $G$ is equal to $M$, then in practice, this verification step terminates almost instantaneously. 

We construct the following subgroup of $J(\F_5)$:
\[ \overline{H} \colonequals  \angle{\overline{G}_1 , \overline{G}_2}.\] This allows us to sieve at $5$ by intersecting with the image of $C(\F_5)$
\[ H' \colonequals \{ c : c \in C(\F_p) \text{ and } (c - \infty) \in \overline{H} \}= 
\{ (1 : 0 : 0), (0 : \pm 1 : 1),(1 :\pm  1 : 1) \},\]
showing that only the reductions of the $\Z_{(5)}$-points modulo $5$ do not fail the Mordell--Weil sieve.

The Chabauty--Coleman method finds points in residue disks corresponding to the $\F_5$-points $\{(2 : 0 : 1), (4 : 4 : 1),(4 : 1 : 1), (3 : 3 : 1), (3 : 2 : 1)\}$, which are ruled out by this test.

The extra points $R_i$ and $W$ found by the Chabauty--Coleman method but not the geometric linear Chabauty method are torsion in $J(\Z_5)/M$ but do not lie in $M$.
\end{example}

\begin{example}
\label{ex-g3-badred}
Finally, we give an example of Theorem \ref{mainthmcomp} (\ref{condition2}) where $M$ does not have good reduction, and the Chabauty--Coleman set contains extra points $R$ such that $\log (R - b)$ is not in $\log M$, only in its $p$-saturation. These extra points pass the Mordell--Weil sieve but are ruled out by the geometric linear Chabauty method.

Let $C/\Z_{(3)}$ be (the smooth projective model of) the genus $3$ curve \[ y^2 = x^7 - 3x^6 + 5x^5 - 5x^4 + 3x^3 - x^2 + 1/4
\] 
taken from a database of genus $3$ hyperelliptic curves over $\Q$ \cite{g3database} computed using the methods in \cite{SutherlandHyperelliptic}.

The points up to height $1000$ are 
\[ C(\Z_{(3)})_{\text{known}} \colonequals \{ (1 : 0 : 0), (0 : -1/2 : 1), (0 : 1/2 : 1), (1 : -1/2 : 1), (1 : 1/2 : 1)\}. \]

According to computations with the fake $2$-Selmer group done by the method \texttt{RankBounds} in Magma, the Mordell--Weil rank of $J$ is $2$.

The Chabauty--Coleman method produces the set
\begin{align*}
&C(\Z_3)_\CC =  \\
&\{ (1:0:0),  (0 : -1/2 : 1), (0 : 1/2 : 1), (1 : -1/2 : 1), (1 : 1/2 : 1), \\
&R_1 \colonequals (2 + 3^3 + 2\cdot3^4 + 2\cdot3^7  +O(3^{8}) : 1 + 3 + 2\cdot3^4 + 2\cdot3^5 + 2\cdot3^6  + O(3^8) : 1),\\ 
&R_2 \colonequals (2 + 3^3 + 2\cdot3^4 + 2\cdot3^7  +O(3^{8}) : -(1 + 3 + 2\cdot3^4 + 2\cdot3^5 + 2\cdot3^6 + O(3^8)) : 1)
 \}.
\end{align*} However, we will use geometric linear Chabauty modulo $p=3$ to show that the residue disks over the reductions $\overline{R_i}$ do not contain any rational points. We consider the residue disk over $\overline{R_1} = (2 : 1 : 1)$; the other will follow by the hyperelliptic involution. Using Magma, we can check the residue disk of the Mordell--Weil group over $\overline{R_1 - b}$ is non-empty, and we can even find an explicit $T \in M_{\overline{R_1 - b}}$, namely $ T \colonequals 73((1 : 0 : 0) - (0 : -1 : 1))$.

Similar to the previous example, as per Remark~\ref{rem-generators-for-M0}, we compute $G_1 \colonequals (0:1/2:1) - \infty$ and $G_2 \colonequals (1:1/2:1) - \infty$ generating a subgroup $H$ of the Mordell--Weil group, and verify that $H$ is saturated at the single prime $p=3$. Since $\overline{R_1}$ passes the Mordell--Weil sieve at $3$, we do not have to check $H$ is saturated at any other primes.

To compute $\Mbar_0$, we find a basis $\langle  \widetilde{G_1} \colonequals -6G_1 - 4G_2$, $\widetilde{G_2} \colonequals -10G_1 + 11 G_2 \rangle$ for the kernel of reduction modulo $3$, and compute
\begin{align*}
\log ( \widetilde{G_1})&= (2, 1, 1)\\
 \log ( \widetilde{G_2})&=(2, 1,1)
\end{align*}
in $\F_3^3$. The map $M_0/3M_0 \to J(\Z / 3^2 \Z)_0$ has a $1$-dimensional kernel generated by $\widetilde{G_1} - \widetilde{G_2}$, so $\Mbar_0$ has bad reduction.

By lifting $\overline{R_1}$ in two different ways, $Q_1$ and $Q_2$, and taking a tiny integral, we compute that the subspace $D_Q = \angle{Q_1 - Q_2} $ is spanned by the vector $\log(Q_1 - Q_2) =(2, 1, 2)$ in $\F_3^3$. Finally,  $\log(v) = \log(Q_1 - \infty - T)=  (2,2,1)$. Altogether, we have computed the determinant zero matrix
\[
 A_{(2 : 1 : 1)} =  \begin{pmatrix} 2 & 2 & 2 \\ 1& 1 & 1 \\ 1 & 1 & 2 \end{pmatrix}
\] representing the linear map $\phi: D_Q \oplus \Mbar_0 \to J(\Z/3 \Z)_0$. 
Since $\log(v)$ is not in the image of this matrix, the residue disk over $\overline{R_1}$ does not contain any rational points. Furthermore, applying the hyperelliptic involution to the calculations, we can also rule out the residue disk of $R_2$ from containing rational points.

Hence $C(\Z_3)_\GC = C(\Z_{(3)}) = C(\Z_{(3)})_{\text{known}}$ does not contain $R_1$ and $R_2$. Another way to see this, by Corollary~\ref{cor-pnmidjfp}, which is applicable as $|J(\F_3)| = 2 \cdot 53$, is to compute the following integral
\begin{align*}
\log(R_1 - \infty)  &= (2 + 2\cdot3 + 2\cdot3^2 +O(3^3), 2 + 3 + O(3^3), 1 + 2\cdot3 + 2\cdot3^2 + O(3^3));
\end{align*}
we see that reduced modulo $3$, this is not in the span of the reductions modulo $3$ of $\log ( \widetilde{G_1})$ and $\log ( \widetilde{G_2})$. We can compute that, truncated to $5$ digits of precision, we have 
\begin{align*}
&\log(R_1 - \infty) =\\
& (2\cdot3^{-1} + 2 + 3^4 + O(3^5)) \log (\widetilde{G_1}) + (3^{-1} + 1 + 3 + 3^2 + 2\cdot3^3 + 2\cdot3^4 +  O(3^5)) \log ( \widetilde{G_2}),
\end{align*}
further explaining why this point is found by the Chabauty--Coleman method but not by geometric linear Chabauty, since $R_1 - \infty$ lies in the $3$-saturation of $M$ inside $J(\Z_3)$, but not in $M$ itself.
\end{example}

\if{false}\begin{example}
Finally, we give an example where $M$ does not have good reduction, and the Chabauty--Coleman set contains extra points $R$ such that $\log (R - b)$ is not in $\log M$, only in its $p$-saturation. These extra points pass the Mordell--Weil sieve but are ruled out by the geometric linear Chabauty method.

Let $C/\Z_{(3)}$ be (the smooth projective model of) the genus $2$ curve (LMFDB label \href{https://www.lmfdb.org/Genus2Curve/Q/36460/a/145840/1}{\texttt{36460.a.145840.1}}) given by the equation \[ y^2 = x^5 - 3/4x^4 + 19/2x^3 - 67/4x^2 + 10x - 2.\]
 The Jacobian of $C$ has Mordell--Weil group $J(\Z_{(p)}) \simeq \Z$, generated by $d \colonequals (1:-1:1) - \\(1:0:0)$. We compute
\begin{align*}
\log d = (3 + 3^2 + 2\cdot 3^3 + 3^4 + 3^5 + 3^6 +  O(3^{7}),3 + 2\cdot3^2 + 3^3 + O(3^{7}))
\end{align*}
so in particular $M$ does not have good reduction.

The Chabauty--Coleman method produces the set 
\begin{align*}
&C(\Z_3)_\CC =  \\
&\{ (1:0:0), (1:-1:1),(1:1:1), \\
&(3^5 + 3^6 + 3^8 + 2\cdot3^{10} + 2\cdot3^{11} + O(3^{12}) : 1 + 3 + 2\cdot3^2 + 3^5 + 2\cdot3^6 + O(3^7) : 1),\\ 
&(3^5 + 3^6 + 3^8 + 2\cdot3^{10} + 2\cdot3^{11} + O(3^{12}) : -(1 + 3 + 2\cdot3^2 + 3^5 + 2\cdot3^6 + O(3^7)) : 1)
 \}.
\end{align*}
We call these latter two points $R_1$ and $R_2$ respectively. Then
\begin{align*}
\log(R_1 - \infty)  &= (1 + 2\cdot3^4 + O(3^5), 1 + 3 + 3^2 + 3^3 + 3^4 + O(3^5))\\
\log (R_2  - \infty) &= (2 + 2\cdot 3 + 2\cdot3^2 + 2\cdot3^3 +O(3^5), 2 + 3 + 3^2 + 3^3 + 3^4 + O(3^5)).
\end{align*}
Hence $R_1$  and $R_2$ are not in $\AJ_\infty (C(\Z_{3}))\cap M$ and the geometric linear Chabauty algorithm successfully rules out these $3$-adic points from consideration as possible rational points.
\end{example}\fi

\newcommand{\etalchar}[1]{$^{#1}$}

\end{document}